\documentclass[12pt]{article}

\usepackage{natbib}            
\usepackage{graphicx}          
\usepackage{amsfonts,amsmath,bm,amsthm}
\usepackage{algorithm}
\usepackage[noend]{algorithmic} 
\usepackage{geometry}
 \usepackage[pdftex,
 	colorlinks,
 	pdftitle={Global optimal feedback control for quantized nonlinear event systems subject to information delays and losses},
 	pdfauthor={Stefan Jerg, Oliver Junge, and Marcus Post},
 	hyperindex,
 	filecolor=blue,
 	urlcolor=blue,
 	citecolor=blue,
 	linkcolor=blue]{hyperref}
                               
\newcommand{\by}{{\bm y}}
\newcommand{\bz}{{\bm z}}
\newcommand{\bs}{{\bm s}}
\newcommand{\bt}{{\bm t}}
\newcommand{\bPhi}{{\bm \Phi}}

\newcommand{\bu}{{\bm u}}
\newcommand{\bx}{{\bm x}}
\newcommand{\bw}{{\bm w}}

\newcommand{\ubu}{\underline{\bm u}}
\newcommand{\ug}{\underline{\gamma}}

\newcommand{\argmin}{\mathop{\rm argmin}}

\newcommand{\R}{\field{R}}
\newcommand{\N}{\field{N}}

\newcommand{\mrm}[1]{\mathrm{#1}}
\newcommand{\argmax}{\mathop{\mathrm{argmax}}}

\newcommand{\field}[1]{\mathbb{#1}}

\newcommand{\cZ}{\mathcal{Z}}
\newcommand{\cE}{\mathcal{E}}

\newcommand{\cU}{{\mathcal U}}

\newcommand{\cP}{{\mathcal P}}

\newcommand{\cQ}{{\mathcal Q}}

\newcommand{\cX}{{\mathcal X}}

\newtheorem{theorem}{Theorem}

\begin{document}

\baselineskip 16pt

\title{Lazy global feedbacks\\ for quantized nonlinear event systems}

\author{Stefan Jerg, Oliver Junge%
\thanks{Faculty for Mathematics, Technische Universit\"at M\"unchen, 85748~Garching, Germany, \textup{\tt  jerg@ma.tum.de, oj@tum.de}}%
\thanks{This work was supported in part by the priority program SPP 1305 of the German research foundation (DFG) and by the Bavarian State Ministry of Sciences, Research and Arts through the PhD program TopMath}}

\date{March 2012}
\maketitle 

\begin{abstract}                          
We consider nonlinear event systems with quantized state information and design a globally stabilizing controller from which only the minimal required number of control value changes along the feedback trajectory to a given initial condition is transmitted to the plant. In addition, we present a non-optimal heuristic approach which might reduce the number of control value changes and requires a lower computational effort. The constructions are illustrated by two numerical examples.
\end{abstract}

\section{Introduction}

Traditionally, controllers for (nonlinear) systems have been designed using a continuum as the underlying time domain.  With the rise of digital information processing, \emph{time-triggered} or \emph{sampled-data} controller designs have become popular. There, a regular grid of time instances serves as the time domain, cf.\ \cite{AsWi97}. Both schemes close the control loop independently of the system's behavior.  This might lead to unnecessary communication between the controller and the plant.  In the case that the communication is implemented via a digital network, restrictions like the maximal bandwith or load dependent stochastic effects might play a role and influence the behavior of the closed loop system.  In order to decrease the network load and possibly avoid these effects, in \emph{event based} control, information is only transmitted  when necessary in order to ensure stability of the closed loop system, see, e.g.,  \cite{Ar99,OtMoTi02a,AsBe02,TaXi06a,KoBr06,VaKa06a,HeSaVa07,As08,GrJu08a,LuLe09a}. 

Another means of reducing the amount of data which has to be transmitted (and thus further reducing the network load) is to use a \emph{quantization} of an underlying continuous state space.  While any real number which is transmitted digitally, necessarily comes from a quantized set since only finitely many digits can be transmitted, here one aims for quantizations which are as coarse as possible since then fewer bits will suffice to encode the data.  We refer to, e.g., \cite{Hs92b,Lu94a,FoJuLu02a,Sch03a}.

Recently, a new approach for the construction of controllers for quantized systems has been proposed which relies on a set oriented approach in combination with graph theoretic techniques, cf.\ \cite{JuOs04a,GrJu07a,GrJu08a}.  In \cite{GrMu09a}, this approach has been extended to event systems.

In the present contribution we describe how to extend this approach such that the number of times that data has to be transmitted from the controller to the plant is minimized along the feedback trajectory to some initial condition.  The construction is based on the optimality principle with a suitably chosen state space and cost function.
Additionally, we present a non-optimal heuristic approach which also reduces the number of data transmission events while requiring a significantly lower computational effort.  The two constructions are illustrated by two numerical examples, an inverted pendulum and a thermofluid batch process.

The paper is structured as follows: In Section~\ref{sec:event} we briefly summarize the basic construction from \cite{GrJu08a,GrMu09a}.  In Section~\ref{sec:mincc} we describe how to extend this construction such that the number of data transmissions from the controller to the plant is minimized and illustrate this construction by two numerical examples in Section~\ref{sec:NumExpMinCC}.  Finally, in Section~\ref{sec:heuristic}, we propose the heuristic scheme for reducing the communication effort.  Again, this is illustrated by the two examples from Section~\ref{sec:NumExpMinCC}.

\section{Global optimal feedbacks for quantized nonlinear event systems}\label{sec:event}

In this section, we summarize the constructions from \cite{GrJu08a,GrMu09a}: We are given a plant which is modeled by a nonlinear discrete time control system (which may, e.g., be derived from a continuous time system by time sampling)
\begin{equation}\label{eq:plant}
\bx(k+1) = f(\bx(k),\bu(k)), \quad k=0,1,2,\ldots,
\end{equation}
where $f:\cX\times \cU\to \mathcal{X}$ is continuous, $\bx(k)\in \cX$ is the state and $\bu(k)\in \cU$ is the control input, $\mathcal{X}\subset\R^n$ and $\cU\subset\R^m$ compact. 
In addition to $f$, we are given a continuous running cost function $c:\mathcal{X}\times \cU\to [0,\infty)$ as well as a
a \emph{target set} $\mathcal{X}^*\subset \mathcal{X}$. We assume $c$ to satisfy $c(\bx,\bu)=0$ iff $\bx\in \mathcal{X}^*$. Our goal is to compute a feedback law for this system which drives the system into the target set $\mathcal{X}^*$ while minimizing the accumulated cost.  However, the information which is transmitted from the plant to the controller is restricted in the following two ways:
\begin{enumerate}
\item \emph{Event model:} The controller only receives information on the state whenever an \emph{event} occurs. Formally, based on the discrete time model (\ref{eq:plant}) of the plant, we are dealing with the discrete time system
\begin{equation}\label{eq:event system}
\bx(\ell+1) = \tilde f(\bx(\ell),\bu(\ell)), \quad \ell=0,1,\ldots,
\end{equation}
where
\begin{equation}
\tilde f(\bx,\bu) = f^{r(\bx,\bu)}(\bx,\bu),
\end{equation}
$r:\mathcal{X}\times \cU\to \N_0$ is a given \emph{event function} and the iterate $f^r$ is defined by $f^0(\bx,\bu)=\bx$ and $f^r(\bx,\bu)=f(f^{r-1}(\bx,\bu),\bu)$, cf.\ \citep{GrMu09a}.  Accordingly, we define an associated running cost $\tilde c:\cX\times\cU\to [0,\infty)$ by
\begin{equation}\label{eq:running cost}
\tilde c(\bx,\bu) = \sum_{k=0}^{r(\bx,\bu)-1} c(f^k(\bx,\bu),\bu).
\end{equation}
Note that we can reconstruct the ``true time'' $k$ from the ``event time'' $\ell$ by the event function $r$: we have that
\[
k(\ell+1)=k(\ell)+r(\bx(\ell),\bu(\ell)).
\]
\item \emph{Quantization model:} The controller only receives \emph{quantized} information on the state.  Formally, we are given a (finite) partition $P=\{\cP_1,\ldots,\cP_d\}$, $\cP_i\subset \mathcal{X}$, of $\mathcal{X}$ which induces an equivalence relation $\sim$ on $\cX\times\cX$ by $x\sim y$ $\Leftrightarrow$ $x$ and $y$ lie in the same partition element.  We denote by $[\bx]\in P$ the corresponding equivalence class of $\bx\in\cX$. Only $[\bx(\ell)]$ is transmitted from the plant to the controller at the event time $\ell$.  Thus, from the viewpoint of the controller, the plant is given by the finite state system, cf.\ \citep{GrJu08a,GrMu09a}
\begin{equation}\label{eq:finite state system}
\cP(\ell+1) = F(\cP(\ell),\bu(\ell),\gamma(\ell)), \quad \ell=0,1,\ldots,
\end{equation}
defined by
\[
F(\cP,\bu,\gamma)=[\tilde f(\gamma(\cP),\bu)],\qquad \cP\in P,\bu\in\cU,
\]
where $\gamma:P\to \mathcal{X}$ denotes a \emph{choice function} which  satisfies $[\gamma(\cP)]=\cP$ for all $\cP\in P$. The choice function models the fact that it is unknown to the controller from which exact state $\bx(\ell)$ the system transits to the next cell $\cP(\ell+1)$. We let $\Gamma$ denote the set of those functions.  Thus, in each step, the dynamics is influenced by the two control parameters $\bu(\ell)$ and $\gamma(\ell)$.  Our goal is to choose $\bu(\ell)$ in each step such that the state $\cP(\ell)$ is controlled into the target set.  At the same time, the influence of the choice function may be viewed as a perturbation which might prevent us from reaching $\cX^*$.  In this sense,  (\ref{eq:finite state system}) constitutes a \emph{dynamic game}, cf.\ \cite{GrJu08a}.  
\end{enumerate}

\subsection{Computing the optimal feedback}

In order to be compatible with our quantization, from now on we assume that $\mathcal{X}^*$ is given by the union of some elements from $P$.  
For the quantized system (\ref{eq:finite state system}) we define
\[
C(\cP,\bu) := \sup_{\bx\in \cP} \tilde c(\bx,\bu).
\]
For given $\cP(0)\in P, \ubu=(\bu(\ell))_\ell\in \cU^\N$ and $\ug=(\gamma(\ell))_\ell\in\Gamma^\N$, the cost accumulated along the associated trajectory $(\cP(\ell))_\ell\in P^\N$ of (\ref{eq:finite state system}) (which depends on $\ubu$ and $\ug$) is
\[
J(\cP(0),\ubu,\ug) := \sum_{\ell=0}^L C(\cP(\ell),\bu(\ell))\in [0,\infty],
\]
where $L=\min\{\ell\geq 0 : \cP(\ell)\subset \cX^*\}\in\N\cup\{\infty\}$.
The \emph{(upper) optimal value function} is 
\[
V(\cP) := \sup_{\Gamma}\inf_{\ubu\in \cU^\N} J(\cP,\ubu,\Gamma(\ubu)) \in [0,\infty],
\]
where $\Gamma:\cU^\N\to\Gamma^\N$ is a strategy of the form
\begin{align*}
\Gamma(\ubu) & = \Gamma((\bu(\ell))_\ell) \\
& = (\gamma_1(\bu(1)),\gamma_2(\bu(1),\bu(2)),\gamma_3(\bu(1),\bu(2),\bu(3)),\ldots)
\end{align*}
and the sup in the definition of the optimal value function is over all strategies of this form.
The optimal value function -- by standard arguments, cf.\ \citep{Be95a} -- for $\cP$ in the stabilizable set $S=\{\cP\in P\mid V(\cP)<\infty\}$ is the unique solution to the \emph{optimality principle}
\begin{equation}\label{eq:value function}
V(\cP) = \inf_{\bu\in \cU}\left\{C(\cP,\bu)+\sup_{\gamma\in\Gamma}V(F(\cP,\bu,\gamma))\right\}
\end{equation}
together with the boundary condition $V(\cP)=0$ for $\cP\subset\cX^*$. Given $V$, we obtain  an optimal feedback for (\ref{eq:finite state system}) by setting
\begin{equation}\label{eq:feedbacklaw}
u(\cP) := \argmin_{\bu\in \cU}\left\{C(\cP,\bu)+\sup_{\gamma\in\Gamma} V(F(\cP,\bu,\gamma))\right\}
\end{equation}
for $\cP\in S$.  Note that we also immediately obtain a feedback for the original system (\ref{eq:plant}) resp.\ (\ref{eq:event system}) from this by setting $u(\bx):=u([\bx])$ for $\bx \in \bigcup_{\cP\in S} \cP$.  By construction, any trajectory of the closed loop system
\begin{equation}\label{eq:closed loop}
\bx(k+1)=f(\bx(k),u([\bx(k)]))
\end{equation}
with $x(0)\in S$ is eventually reaching the target set $\cX^*$, cf.\ \cite{GrJu07a,GrMu09a}.

Computationally, the feedback is constructed using a directed, weighted hypergraph together with a corresponding shortest path algorithm, cf.\ \cite{JuOs04a,GrJu08a,Lo07a} (see also Section~\ref{sec:heuristic}).

\section{Construction of a lazy feedback}
\label{sec:mincc}

When the data transmission between the plant and the controller is realized via a digital network it is often desirable to minimize the amount of transmitted information in order to reduce the overall network load.  More specifically, here we treat the question of how to minimize the number of times that a new control value has to be transmitted from the controller to the plant.  Using an optimization based feedback construction, this goal can directly be modelled by suitably defining the running cost function.

In order to detect a change in the control value generated by the controller we need to be able to compare to its value from the previous time step (resp. event).  We therefore define  the \emph{extended state space} $\cZ := \cX \times \cU$. Based on the event system \eqref{eq:event system}, we consider the event system 
\begin{equation}\label{eq:mincc system}
\bz(\ell+1) = g(\bz(\ell),\bu(\ell)), \quad \ell=0,1,2,\ldots
\end{equation}
with $\bz(\ell)=(\bx(\ell),\bw(\ell))\in\cZ$ the extended state vector, $\bu(\ell) \in \cU$ and $g:\cZ\times\cU\to\cZ$ is defined by
\[
g(\bz,\bu)=g((\bx,\bw),\bu)=\left[\begin{array}{c} \tilde f(\bx,\bu) \\
\bu\\
\end{array}\right].
\]
We define $\cZ^*:=\cX^*\times \cU$ as the target set in the extended state space so that reaching the target only depends on the state $\bx$. We further define an associated running cost function $d: \cZ \times \cU \rightarrow [0,\infty)$ by
\begin{align}\label{eq:mincc costs}
d((\bx,\bw),\bu) & =  (1-\lambda) \tilde c(\bx, \bu) + \lambda (1-\delta(\bu-\bw))
\end{align}
with 
\begin{equation}
\label{eq:mincc varphi}
\delta(\bu) := \left\{\begin{array}{ll} 1, & \text{if } \bu = 0,\\
0, & \text{else.} \end{array}\right.
\end{equation}
Here, $\lambda \in [0,1)$ must be strictly $<1$ in order to guarantee that $d(\bz,\bu) = 0$ iff $\bz \in \cZ^*$. Note that the dynamics of the extended system (\ref{eq:mincc system}) does not depend on the second component $\bw$ of the extended state vector $\bz$, but  only the modified cost function $d$ does.

We can now apply the construction from the previous section to the system~(\ref{eq:mincc system}) with cost function~(\ref{eq:mincc costs}).  To this end, we would need to construct a partition of $\cU$.  Instead, in order to simplify the exposition, here we simply assume that $\cU$ is discrete, i.e.\ contains only finitely many elements.  We then use $P\times \cU$ as the underlying partition for the quantization.  We denote the resulting optimal value function by $V_\lambda:\cZ\to[0,\infty]$, the stabilizable subset by
\[
S_\lambda:=\{\bz \in \cZ: V_\lambda([\bz]) < \infty\}
\]
and the associated feedback by $u_\lambda:S_\lambda\to\cU$.   

\medskip

We will show that for a sufficiently large $\lambda < 1$ the closed loop system
\begin{equation}\label{eq:extended closed}
	\bz(\ell+1) = g(\bz(\ell),u_\lambda(\bz(\ell))),  \quad \ell=0,1,2,\ldots,
\end{equation}
is asymptotically stable on $S_\lambda$, i.e.\ that for $\bz(0)\in S_\lambda$ the associated trajectory enters $\cZ^*$ in finitely many steps. Furthermore, the number of control value changes along this trajectory will be minimal.  

To be more precise: For some initial state $\bz=\bz(0)\in S_\lambda$ let $(\bz(\ell))_\ell\in \cZ^\N$, $\bz(\ell)=(\bx(\ell),\bw(\ell)))$, be the corresponding trajectory of the closed loop system (\ref{eq:extended closed}), let
\[
	L(\bz,u_\lambda) = \min\{\ell \geq 0:\bz(\ell) \in \cZ^*\}
\]
be the number of time steps until the trajectory reaches the target set $\cZ^*$, 
\[
	E(\bz,u_\lambda) = \sum_{\ell=0}^{L(\bz,u_\lambda)} 1-\delta\bigl(u_\lambda(\bz(\ell))-\bw(\ell)\bigr)
\]
the number of control value changes along the corresponding trajectory as well as
\[
\tilde{J}(\bz,u_\lambda) = \sum_{\ell=0}^{L(\bz,u_\lambda)} \tilde{c}(\bx(\ell),u_\lambda(\bz(\ell)))
\]
the accumulated (original) costs.

\begin{theorem}
For all $\lambda \in [0,1)$, $S \times \cU\subset S_\lambda$ 
and $\bx(\ell) \rightarrow \cX^*$ as $\ell\to\infty$.

Further, there exists a $\lambda < 1$ such that for any feedback $u:S_\lambda\to\cU$ for the extended system and $\bz\in S_\lambda$ with $L(\bz,u) < \infty$  holds $E(\bz, u) \geq E(\bz, u_\lambda)$.
\end{theorem}

\begin{proof}
By definition, the extended system \eqref{eq:mincc system} and the  cost function \eqref{eq:mincc costs} fulfill the assumptions in \cite{GrJu07a}, so asymptotic stability of the closed loop system (\ref{eq:extended closed}) directly follows for all $\bz \in S_\lambda$ by their proof.

In order to show that $S\times\cU\subset S_\lambda$ for all $\lambda\in [0,1)$, choose $\lambda\in [0,1)$ and some initial value $\bz(0) = (\bx(0), \bu(0))\in S\times\cU$ arbitrarily. Consider the feedback
\[
u(\bz) =u((\bx,\bu)):= u(\bx)
\]
for system \eqref{eq:mincc system}, where $u(\bx)$ denotes the feedback for (\ref{eq:event system}) which has been constructed in Section~\ref{sec:event}.  This leads to a trajectory $(\bx(\ell),\bu(\ell))_\ell$ of the extended system with $(\bx(\ell))_\ell$ being exactly the trajectory of the original system (\ref{eq:event system}).
Since $\bx(0) \in S$, $V(\bx(0))$ is finite and the accumulated cost $\tilde J(\bz(0),u)$ for this trajectory does not exceed $(1-\lambda)V(\bx(0)) + \lambda L(\bz(0),u)$ which is finite.
According to the optimality of $V_\lambda$,
\[
V_\lambda(\bz(0)) \leq (1-\lambda)V(\bx(0)) + \lambda L(\bz(0),u) <\infty
\]
follows, i.e. $\bz(0) \in S_\lambda$.

To show the optimality of $u_\lambda$ with respect to the number of control value changes, assume there exists a feedback $\bar{u}:S_\lambda\to\cU$ with $E(\bz,\bar{u}) \leq E(\bz,u_\lambda)-1$ for some $\bz \in S_\lambda$.
Since $u_\lambda$ is optimal, the following inequality holds:

\begin{align}
	       (1-\lambda)\tilde{J}&(\bz,\bar{u})     + \lambda E(\bz,\bar{u}) \\
	 &\geq (1-\lambda)\tilde{J}(\bz,u_\lambda) + \lambda E(\bz,u_\lambda) \\
	 &\geq (1-\lambda)\tilde{J}(\bz,u_\lambda) + \lambda E(\bz,\bar{u}) + \lambda 
\end{align}
and thus
\begin{equation*}
(1-\lambda)\tilde{J}(\bz,\bar{u}) \geq (1-\lambda)\tilde{J}(\bz,u_\lambda) + \lambda.
\end{equation*}
As the costs denoted by $\tilde{J}$ are finite, $\lambda \rightarrow 1$ leads to a contradiction.
\end{proof}


\section{Numerical experiments}
\label{sec:NumExpMinCC}

\subsection{Nonlinear Inverted Pendulum}

For our numerical experiments we first consider an inverted pendulum on a cart, cf.\ \citep{JaYuHa01a}. The motion of the pendulum is given by the continuous time control system
\[
        \left(\frac{4}{3} - m_r \cos^2 \varphi\right) \ddot\varphi +
        \frac{m_r}{2} \dot\varphi^2 \sin 2 \varphi -
        \frac{g}{\ell}\sin \varphi  
  =  - u\;\frac{m_r}{m 
        \ell} \cos \varphi, 
\]
where $(\varphi,\dot\varphi)\in [0,2\pi]\times\R$ denotes the state of the pendulum and $u\in\cU\subset\R$ is the control input.  We have used the parameters $m=2$ for the pendulum mass, 
$m_r = m / (m + M)$ for the mass ratio with cart mass $M=8$, 
$\ell = 0.5$ as the length of the pendulum and $g = 9.8$ for the
gravitational constant. As the instantaneous cost function, we employ  
\begin{equation}
\label{eq:costcont}
        q(\varphi,\dot\varphi,u,t) = \frac{0.01}{2} u^2 + t,
\end{equation}
where $t\in\R$ is the system's time. Denoting the system's evolution operator for constant control functions $u(t)\equiv\bu\in\cU$ by
$\bPhi^t(\bx,\bu)$, $\bx=(\varphi,\dot\varphi)$, we consider the discrete time system
$f(\bx,\bu) = \bPhi^T(\bx,\bu)$ for $T=0.01$, i.e., the sampled continuous time system
with sampling rate $T=0.01$. The map $\bPhi^T$ is approximated via the classical Runge-Kutta scheme
of order $4$ with 5 equidistant steps. The discrete time cost function is obtained by
numerically integrating the continuous time instantaneous cost
according to $c(\bx,\bu) = \int_0^T q(\bPhi^t(\bx,\bu),\bu,t)dt$.
We choose the state space $\cX = [-10,10] \times [-8,8]$ and a partition of $2^7 \times 2^7$ equally sized partition elements on $\cX$.
The control space is chosen as $\cU = [-64,64]$, discretized by 17 equidistant samples and
the target region is set to $[-\frac{5}{8},\frac{5}{8}] \times [-\frac{1}{2},\frac{1}{2}]$ (i.e.\ $8 \times 8$ partition elements around $[0,0]$).

Let $\bs(\bx)\in\R^2$ and $\bt(\bx)\in\R^2$ denote the center and the radius of the rectangular partition element
containing $\bx$, respectively, then we define an \emph{event set} via
\begin{equation}
	\beta(\bx) = \{\by=(y_1,y_2)^T \in \cX ~:~ |y_i - s_i(\bx)| \leq e_r t_i(\bx), i=1,2 \}
\end{equation}
with \emph{event radius} $e_r = 9$ and the event function
\begin{equation}
	r(\bx,\bu) = 	
 	\begin{cases}
 		\min \{t \in 0.01 \mathbb{N}  ~:~ \bPhi^t(\bx,\bu) \notin \beta(\bx) \} ,\\
 		\qquad ~\text{if not empty} \\
 		0 , ~\text{else.}
	\end{cases}
\end{equation}
\begin{figure}[htb]
\begin{center}
	\includegraphics[width=0.8\columnwidth]{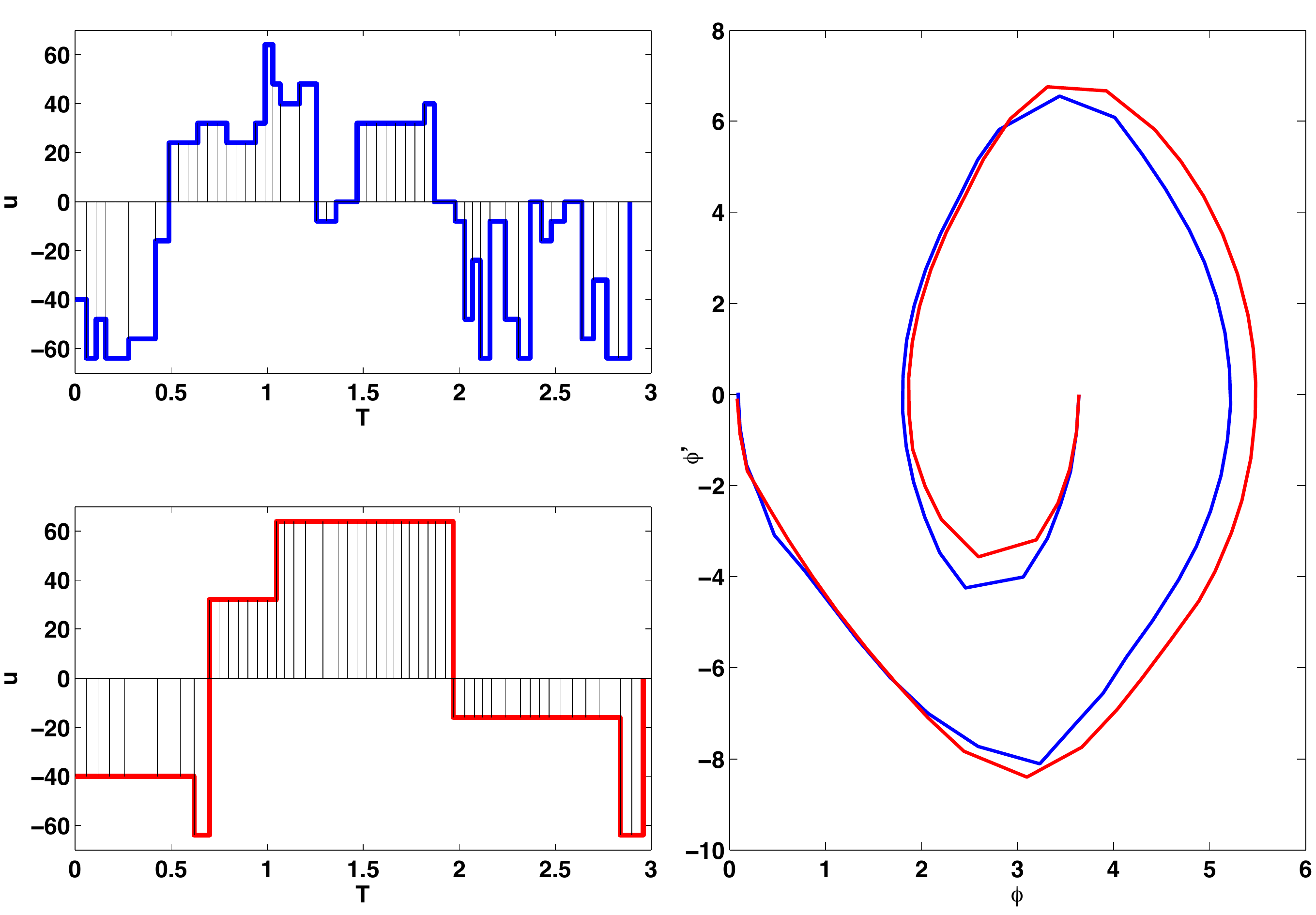}
	\caption{Inverted pendulum: Control sequence over time (left) and the associated trajectories in state space for the initial state $(\pi+0.5,0)^T$.  Blue (dark): generated by the ordinary feedback, red (light): generated by the lazy feedback.}
	\label{fig:MinCC}
\end{center}
\end{figure}
For this quantized event system, we employ the construction in the two previous sections, i.e.\ we compute the (``ordinary'') feedback as described in Section~\ref{sec:event} as well as the lazy feedback  from Section~\ref{sec:mincc} with $\lambda=0.99$.

Figure \ref{fig:MinCC} shows two trajectories of the closed loop system (\ref{eq:closed loop}) starting at the initial state $(\pi+0.5,0)^T\in\cX$. On the right-hand side we show these two trajectories in state space while on the left-hand side we plot the associated control sequences over time. The blue (dark) trajectory results from the ordinary feedback, while the red (light) trajectory is generated with the lazy feedback. 

The gray ticks on the left indicate all time instances where an event occurs.  While the ordinary feedback leads to a change of the control value at almost every event and generates a total number of $32$ control value changes, the lazy feedback stabilizes the system using only $5$ control changes. Note that the time until the target set is reached (which is essentially minimized here) remains almost the same.

\subsection{Batch Reactor}

In this numercial experiment, the aim is to control a thermofluid process in a batch reactor (cf.\ Figure~\ref{batchprocess}) as described in \cite{GrJeJuLeLuMuPo09a}.
The main part of the process consists in the cylindrical batch reactor~$\mrm{TB}$
which has a continuously adjustable inflow via valve~$\mrm
V_1$ of water from the spherical tank~$\mrm T3$ above.
In addition, a permanent outflow only depending on the fluid level
in~$\mrm{TB}$ is present. 
\begin{figure}[ht]
    \centering
    \includegraphics[width=0.7\columnwidth,height=0.5\columnwidth]{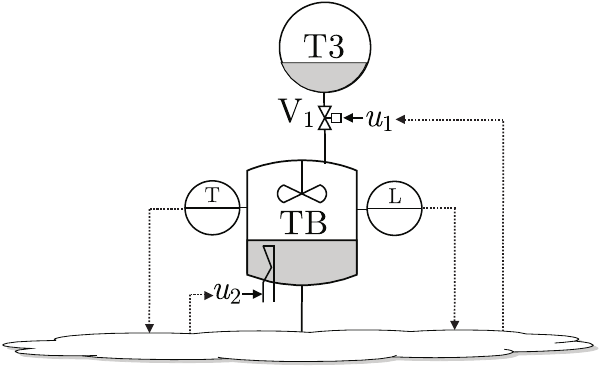}
    \caption{Thermofluid process}
    \label{batchprocess}
\end{figure}
Heating rods can increase the temperature of the fluid
in~$\mrm{TB}$ while cooling can only be achieved by the inflow of cool water from~$\mrm T3$.
The two state variables that can be continously measured are
the fluid level $l_\mrm{TB}$ and the temperature ~$\vartheta_\mrm{TB}$  in~$\mrm{TB}$, so the state becomes $\bm x = (x_1,x_2) = (l_\mrm{TB}, \vartheta_\mrm{TB})$.
The valve angle~$u_1\in [0,1]$ of valve~$\mrm V_1$ and the power~$u_2\in \{0,\ldots,6\}$ of the heating rods are considered
as input $\bm u=(u_1,\;u_2)$.
To model the nonlinear dynamics of the process we use the following differential equation with parameters from
Table~\ref{tab: 1} (cf. \cite{GrJeJuLeLuMuPo09a}):

\begin{equation}\begin{split}
\dot{x}_1 &= \frac{1}{A_\mrm h}
  \left(q_\mrm{T3}(u_1) - K_\mrm{A}\sqrt{2g x_1} \right) \\
\dot{x}_2 &= \frac{1}{V(x_1)} \left( q_\mrm{T3}(u_1)(\vartheta_\mrm{T3}-x_2)
  +\frac{P_\mrm{el}k_\mrm{h}u_2}{\varrho\;c_\mrm{p}}\right),
\end{split}\end{equation}
where
\begin{equation}\begin{split}
q_\mrm{T3}(u) &= \left\{%
\begin{array}{ll}
    7\!\cdot\!10^{-6}(11.1 u^2+13.1 u+0.2) & \hbox{ for } u>0.2,
     \\
    0 & \hbox{ else}\\
\end{array}%
\right.
\\
V(x) &= 0.07 x-1.9\!\cdot\!10^{-3} \; \hbox{for } x>0.26.
\end{split}\end{equation}
The unit of the flow $q_\mrm{T3}$ is m$^3$/s and the unit of the volume
$V$ of the fluid in TB is m$^3$.

\begin{table}[!h]%
   \centering%
   \caption{Parameters and constants}%
   \label{tab: 1}
   \begin{tabular}{lll}
      \hline\noalign{\smallskip}
      Parameter & Value & Meaning \\
      \noalign{\smallskip}\hline\noalign{\smallskip}
      $P_\mrm{el}$ & 3000 W & Electrical power \\
      $k_\mrm h$ & $0.84\; \mrm{J} / (\mrm{W} \, \mrm{s})$ & Heat transfer coefficient \\
      $c_\mrm p$ & $4180\; \mrm{J} / (\mrm{kg} \, \mrm{K})$ & Heat capacity of water \\
      $g$ & $9.81\; \mrm{m} / \mrm{s}^2$ & Gravitation constant \\
      $\varrho$ & $998\; \mrm{kg} / \mrm{m}^3$ & Density of water \\
      $\vartheta_\mrm{T3}$ & $293.15$ K  & Temperature of inflow \\
      $K_\mrm{A}$ & $1.59\!\cdot 10^{-5}$ $\mrm m^3 / \mrm m$ & Outflow parameter \\
      $A_\mrm{h}$ & $0.07 \mrm m^2$ & Cross sectional area \\
      \noalign{\smallskip} \hline
   \end{tabular}
\end{table}

The goal is to steer the system into a neighbourhood of 
the operating point $(\bar{l}_\mrm{TB}, \bar{\vartheta}_\mrm{TB})  =  (0.349 \, \mrm{m}, 310.56 \, \mrm K)$ with only a minimum number of control changes, therefore we use a weighted sum of time and a quadratic function in $\bx$ and $\bu$ with a small gain as cost function.
According to the physical limitations of the reactor, we set the state space to $\cX = [0.26,0.45]\mrm m \times [293.15,323.15] \mrm K$ and use a discretization of $2^6 \times 2^6$ partition elements. The continuous input $u_1\in [0,1]$ is discretized via $12$ equidistant samples. For the time integration of the ordinary differential equation, we use the classical Runge-Kutta scheme of order $4$ with $5$ equidistant steps and a time step of $1 \mrm{s}$. An event is generated whenever the state leaves a partition element, i.e.\ we employ an event radius of $e_r=1$ here.


We consider the initial state $[0.275 \mrm{m},295 \mrm K]$ and -- like in the first example -- compare two trajectories associated, respectively, to the ordinary feedback described in Section~\ref{sec:event} and the lazy one as described in Section~\ref{sec:mincc} with $\lambda=0.9$.  In Figure~\ref{fig:MinCCBatch} (left) we compare the generated control sequences.  Using the lazy feedback, the number of control value changes is reduced dramatically from $13$ to only $2$.  Note that, again, the time required for the system to reach the target set remains almost the same, cf.\ the right part of Figure~\ref{fig:MinCCBatch}.

\begin{figure}[htb]
\begin{center}
	\includegraphics[width=0.98\columnwidth]{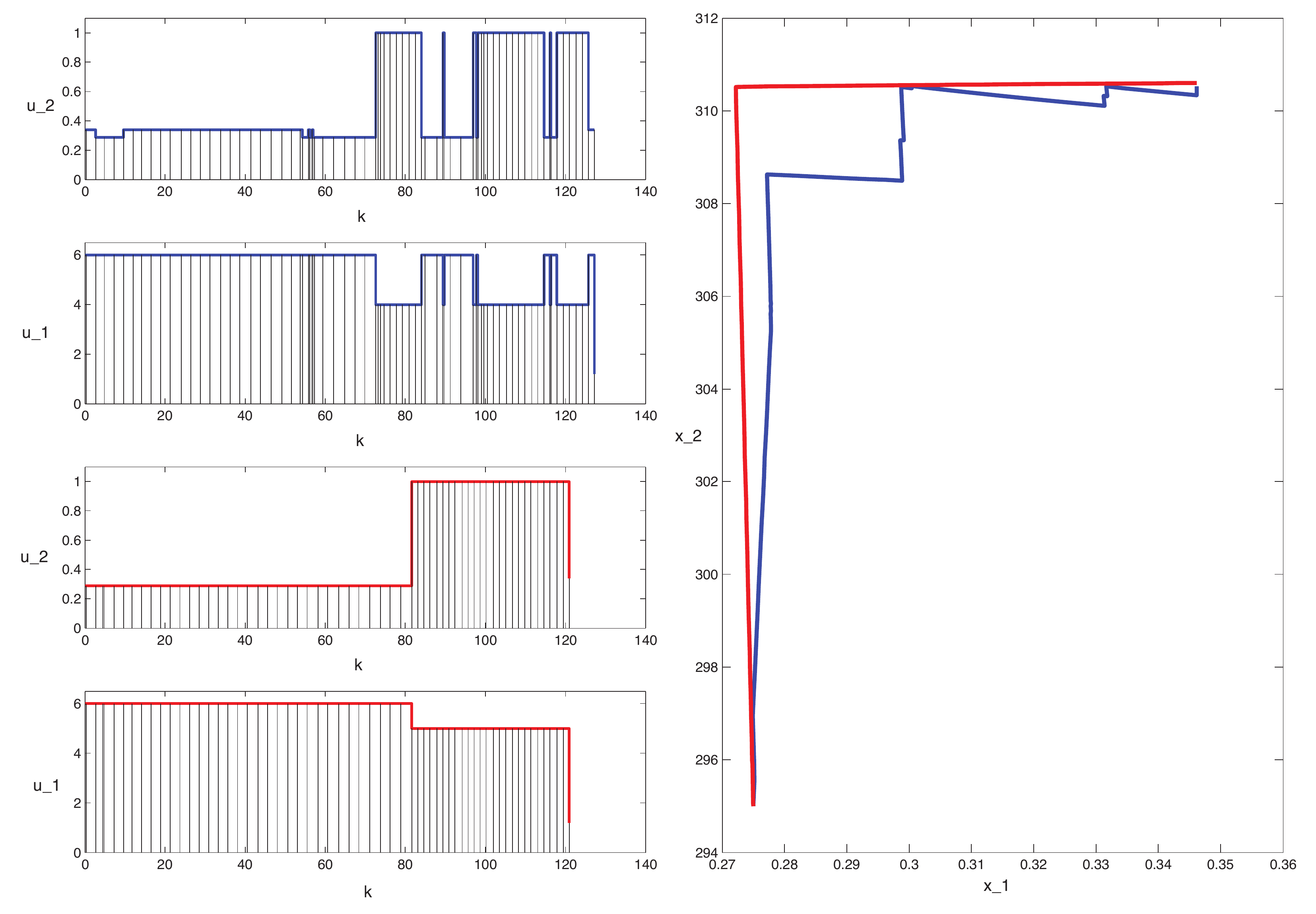}
	\caption{Left: Output of the controller over time for the batch process with initial state $[0.275 \mrm{m},295 \mrm K]$. Upper figure: using the standard cost function. Lower figure: lazy feedback, i.e.\ minimal number of control value changes ($\lambda=0.9$). Right: Associated feedback trajectories in state space. Blue (dark): using the standard cost function. Red (light): lazy feedback, i.e.\ minimal number of control value changes ($\lambda=0.9$).}
	\label{fig:MinCCBatch}
\end{center}
\end{figure}

\section{A heuristic approach for reducing the data transmission frequency}
\label{sec:heuristic}

One drawback of the lazy feedback construction proposed in Section~\ref{sec:mincc} is the need to extend the state space to $\cX \times \cU$. This leads to a notable increase in both memory requirements and computing time. In this section we will illustrate that this expansion is really needed in order to obtain the \emph{minimum} number of control changes.  Nevertheless, there is a heuristic way of \emph{reducing} the number of control changes without these drawbacks.  In order to derive this heuristic, we first have a closer look at the graph construction which is used in order to compute the value function and the feedback, cf.\ \cite{JuOs04a,GrJu08a}.

For each partition element $\cP\in P$, the value  $V(\cP)$ is given by the length of the shortest path from $\cP$ to the set of partition elements which constitute $\cX^*$ in the following hypergraph $G=(P,\mathcal{E})$:  The edges of $G$ are given by
\[
\mathcal{E}=\left\{(\cP,F(\cP,\bu,\Gamma)) \mid \cP\in P, \bu\in\cU\right\}
\]
weighted by
\[
w(\cP,N)=\inf\{C(\cP,\bu): \bu\in\cU, F(\cP,\bu,\Gamma)=N\},
\]
where $F(\cP,\bu,\Gamma):=\{F(\cP,\bu,\gamma)\in P\mid \gamma\in\Gamma\}$.
As such, it can be computed by an efficient Dijkstra-type algorithm \citep{GrJu07a,Lo07a}, cf.\ Algorithm~1.

\begin{algorithm}[H]
\label{alg:Dijkstrahypergraph}
\caption{MinMax-Dijkstra}
\emph{Input:~~}
\begin{tabular}[t]{l}
 	hypergraph $(P,\cE)$\\
 	weights $w:\cE\to (0,\infty)$\\
 	set of target nodes $O\subset P$
\end{tabular}\\
\emph{Output:}
\begin{tabular}[t]{l} 	
	value function $V:P\to [0,\infty]$\\
 	control input $u:P\to\cU$
\end{tabular}\\
\begin{algorithmic}[1]
\setcounter{ALC@line}{0}
	\FORALL{$\cP \in P \backslash O$}
		\STATE $V(\cP) := \infty$
	\ENDFOR
\end{algorithmic}
\begin{algorithmic}[1]
\setcounter{ALC@line}{2}
	\FORALL{$\cP \in O$}
		\STATE $V(\cP) := 0$
	\ENDFOR
\end{algorithmic}
\begin{algorithmic}[1]
\setcounter{ALC@line}{4}
	\STATE $Q := P$\
\end{algorithmic}

\begin{algorithmic}[1]
\setcounter{ALC@line}{5}
\WHILE{$Q \neq \emptyset$}
	\STATE $\cP := \argmin_{\cP' \in Q} V(\cP')$
	\STATE $Q := Q \backslash \{\cP\}$
	\FORALL{$(\cQ,N) \in \cE ~ \text{with} ~ \cP \in N$}
		\IF{$N \subset P \backslash Q$}
			\IF{$V(\cQ) > w(\cQ,N) + V(\cP)$}
					\STATE $V(\cQ) := w(\cQ,N) + V(\cP)$
					\STATE $u(\cQ) := u(\cQ,N)$
			\ENDIF
		\ENDIF
	\ENDFOR
\ENDWHILE
\end{algorithmic}
\end{algorithm}



Here, $u(\cQ,N)=\argmin_{\bu\in\cU} \{C(P,\bu)\mid F(\cQ,\bu,\Gamma)=N\}$.
One easily shows that when a node $\cP$ is removed from $Q$ in line 6, $V(\cP)$ and $u(\cP)$ are fixed until termination (cf.\ for example \citep{AhRa90}).
It follows that when a hyperedge $(\cQ,N)$ is being processed in lines 9 and 10, the value $V(\cP)$ is fixed for all target nodes $\cP \in N$ of the hyperedge because $N \subset P \backslash Q$~\emph{(line~$7,8$)}.
This information can be used for the choice of a proper control $u(\cQ)$.

To be more precise, let $u(\cQ,N)$ denote the control applied for hyperedge $(\cQ,N)$, $u_0(\cP)$ an arbitrary control for the target nodes $\cP \in O$ and $\lambda \in [0, 1)$ a parameter. Then in order to reduce the number of control changes we change lines~$3-4$ and lines~$11-13$ in Algorithm~1 as follows:

\begin{algorithm}
\label{alg:DijkstrahypergraphHeuristic}
\begin{algorithmic}[1]
\setcounter{ALC@line}{2}
\FORALL{$\cP \in O$}
	\STATE $V(\cP) := 0 ~\text{and}~u(\cP) := u_0(\cP)$
\ENDFOR
\end{algorithmic}
~
\begin{algorithmic}[1]
\setcounter{ALC@line}{10}
\IF{$V(\cQ) > (1-\lambda)w(\cQ,N) + \lambda \sigma(u(\cQ,N), N) + V(\cP)$}
	\STATE $V(\cQ) := (1-\lambda)w(\cQ,N) + \lambda \sigma(u(\cQ,N), N) + V(\cP)$
	\STATE $u(\cQ) := u(\cQ,N)$
\ENDIF
\end{algorithmic}
\end{algorithm}


Here, $\sigma:\cU \times P^{|N|} \rightarrow [0,\infty)$ is a function dependent on the control of the hyperedge and the controls of the possible subsequent states. In order to reduce the number of control changes along trajectories, $\sigma(u, N)$ has to be large when $u \neq u(\cP)$ for many $\cP \in N$ and should be small otherwise.
For example, one can set
\[
\sigma(u,N) = 1-\delta\left(u\left(\argmax_{\cP \in N}V(\cP)\right)-u(\cQ,N)\right)
\]
with $\delta$ defined in \eqref{eq:mincc varphi}.  In our numerical tests, however, choosing
\[
\sigma(u,N) = \frac{1}{|N|} \sum_{\cP \in N} 1-\delta\Bigl(u(\cP)-u(\cQ,N)\Bigr)
\]
seemed to lead to better results.

\subsection{A counterexample}

In order to show that the heuristic approach in general does not lead to a feedback which produces the minimal number of control value changes,
we consider the following counterexample, cf.\ Figure \ref{fig:CounterExample}.

\begin{figure}[htb]
\begin{center}
	\includegraphics[width=0.6\columnwidth]{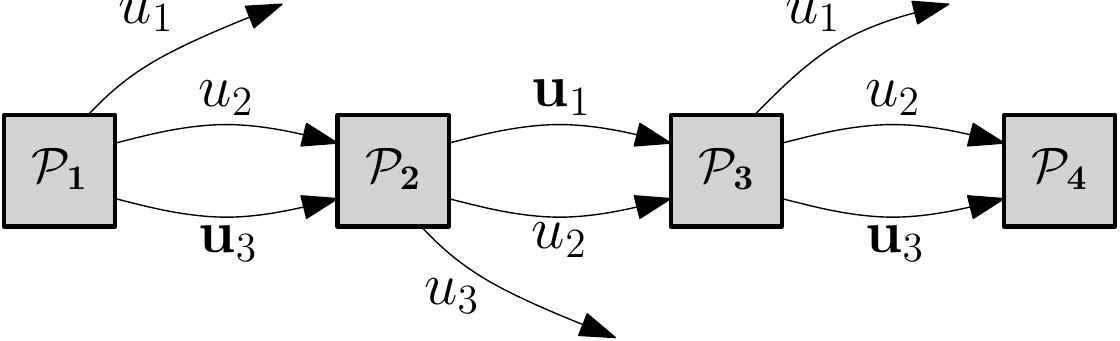}
	\caption{Counterexample to illustrate that the heuristic approach in general does not lead to a feedback which minimizes
	the number of control changes}
	\label{fig:CounterExample}
\end{center}
\end{figure}

We start with node $\cP_4$ on the right and assume that control input $u_1$ is optimal for node $\cP_4$. Let $\cP_3$ be the neighbour of $\cP_4$ being processed next in Algorithm~1. As control input $u_1$ does not lead to a transition to $\cP_4$ we only have the choice between $u_2$ and $u_3$. Since both of them lead to a control change, we may chose as well $u_3$.  This procedure can be done recursively up to state $\cP_1$ on the left with $u_1$ and $u_2$ or $u_1$ and $u_3$, respectively. We end up with a trajectory from $\cP_1$ to $\cP_4$ switching control $3$ times, which is not the optimum compared to having constant control $u_2$ and switching control only once at the end.

The reason for this effect is that for each node $\cP$ we only have computed the optimal value $V(\cP)$ and control input $u(\cP)$ with respect to a trajectory that starts at $\cP$. But we do not know about optimality when $\cP$ is not the initial state of a trajectory. A list of the values $V(\cP)$ for all pairs $(\cP,\bu) \in P \times \cU$ would be sufficient to remedy this -- but this then leads directly back to the lazy feedback approach of section \ref{sec:mincc}.

\subsection{Numerical experiments}
\label{subsec:NumExpRedCC}

Nevertheless, in numerical experiments we often obtain reasonable results even with the heuristic approach. The number of control changes often may be reduced without enlarging the state space and so without any noticeable effect on the memory consumption (cf.\ Figure~\ref{fig:RedCC}, where the number of control value changes was reduced from 33 to 18).
\begin{figure}[htb]
\begin{center}
	\includegraphics[width=0.98\columnwidth]{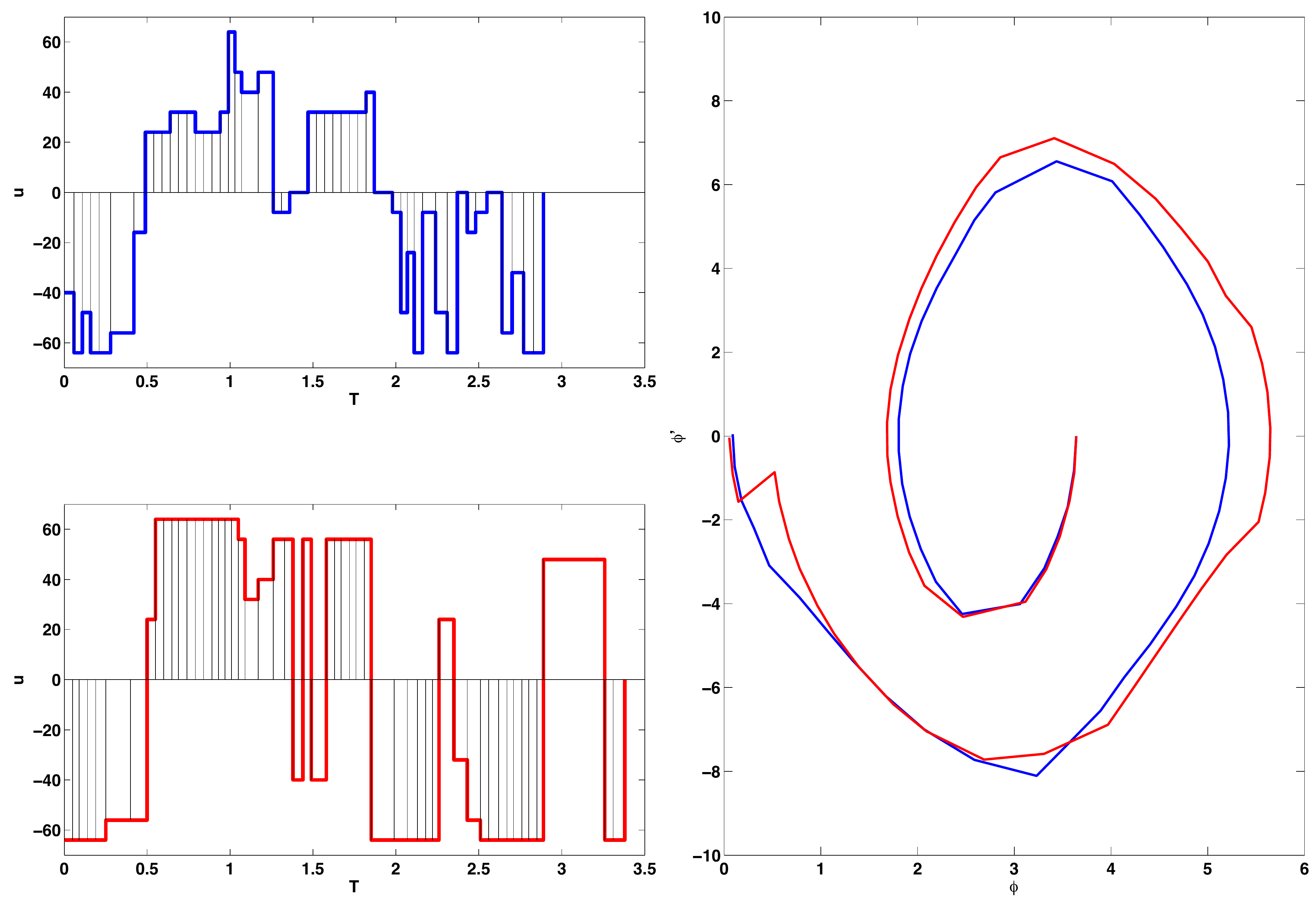}
	\caption{Inverted pendulum: Control sequences (left) and state space trajectory (right)
	 for the ordinary feedback (blue/dark) in comparison to the heuristic approach
	for reducing the number of control changes (red/light).  Parameters are the same as in Section~\ref{sec:NumExpMinCC}.}
	\label{fig:RedCC}
\end{center}
\end{figure}

On the other hand, the heuristic approach does not seem to be prone to failure, as shown in repeating the numerical experiment with the batch reactor.  Here (cf.\ Figure~\ref{fig:RedCCBatch}), the number of control value changes is actually increased.  A closer inspection seems to reveal a phenomenon similar to the counterexample described above.
\begin{figure}[htb]
\begin{center}
	\includegraphics[width=0.7\columnwidth]{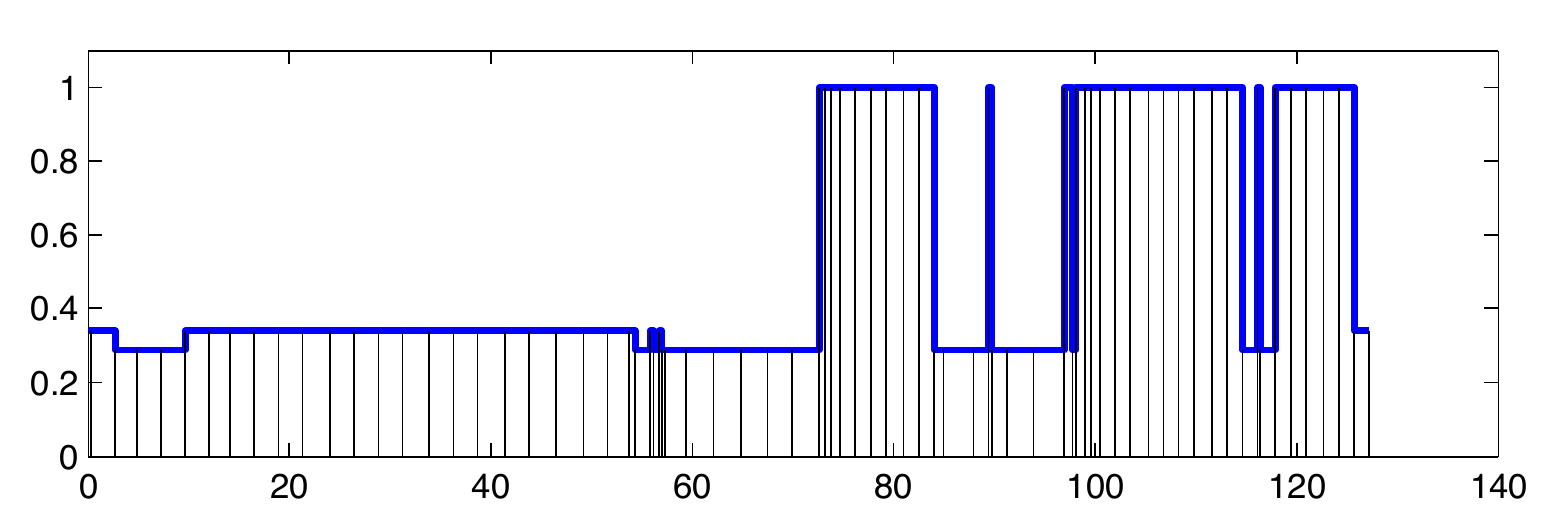}
	\includegraphics[width=0.7\columnwidth]{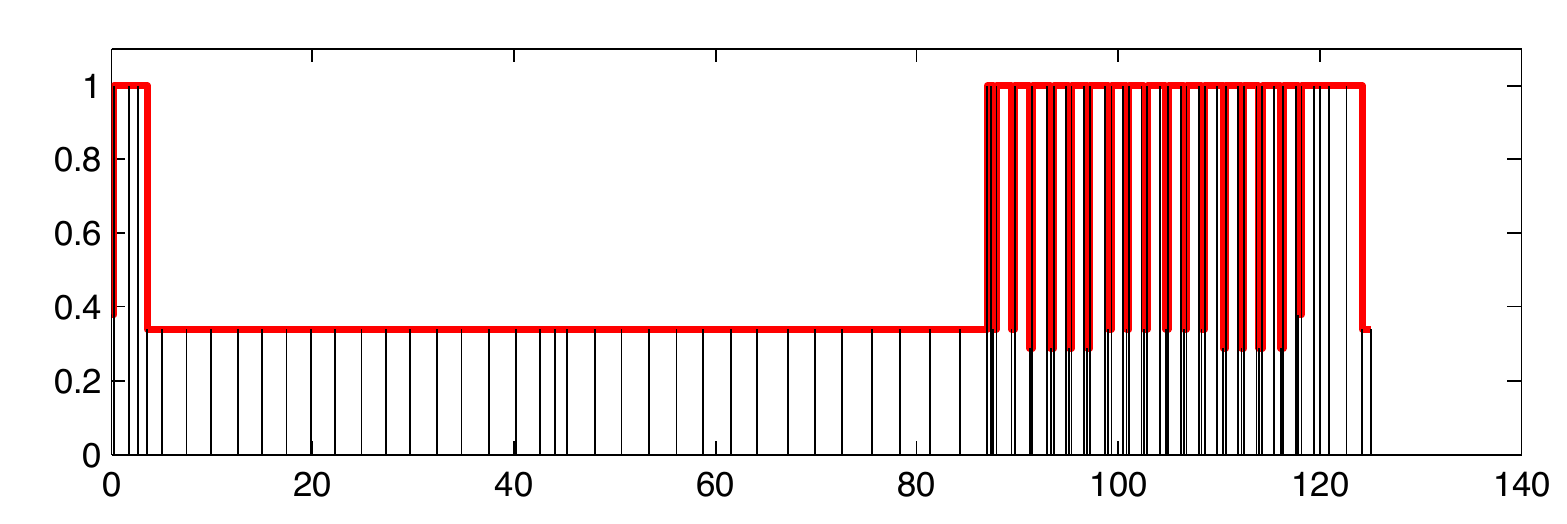}
	\caption{Batch reactor: Control sequences (left) and state space trajectory (right)
	 for the ordinary feedback (blue/dark) in comparison to the heuristic approach
	for reducing the number of control changes (red/light).  Parameters are the same as in Section~\ref{sec:NumExpMinCC}.}
	\label{fig:RedCCBatch}
\end{center}
\end{figure}

\section{Conclusion}

Based on an extended model of a nonlinear quantized event system we introduced the concept of a lazy feedback which stabilizes the system with the minimal number of control value changes.  We show that the lazy feedback is indeed optimal in this sense among all feedbacks which stabilize the same set of initial conditons.
In addition, we illustrated a heuristic method which sometimes can be used for reducing the number of control changes without any further computational effort compared to the lazy feedback construction. 

\bibliographystyle{alpha}
\bibliography{JeJu_lazy}

\end{document}